\documentclass[12pt, eqno, twoside]{article}
\usepackage{graphicx}
\usepackage{amscd}
\usepackage[matrix,arrow,curve]{xy}
\usepackage{amssymb,amsmath}
\usepackage{amsmath,amsfonts,amssymb,latexsym}
\usepackage{fancyhdr}
\newtheorem{example}{Example}[section]
\newtheorem{theorem}{Theorem}[section]
\newtheorem{definition}{Definition}[section]
\newtheorem{proposition}{Proposition}[section]
\newtheorem{corollary}{Corollary}[section]
\newtheorem{lemma}{Lemma}[section]

\newenvironment{proof}[1][Proof]{\noindent\textbf{#1.} }{\
\rule{0.5em}{0.5em}}

\begin{document}

\pagestyle{fancy}
\fancyhead{} 
\fancyhead[EC]{\small\it  Patrice P.
Ntumba, \ Adaeze C. Orioha}%
\fancyhead[EL,OR]{\thepage} \fancyhead[OC]{\small\it Biorthogonality
in $\mathcal A$-Pairings
and Hyperbolic Decomposition Theorem}%
\fancyfoot{} 
\renewcommand\headrulewidth{0.5pt}
\addtolength{\headheight}{2pt} 

\title{\Large{\textbf{Biorthogonality in $\mathcal A$-Pairings and
Hyperbolic Decomposition Theorem for $\mathcal A$-Modules}}}
\author{Patrice P.
Ntumba\footnote{Is the corresponding author for the paper.}, \
Adaeze C. Orioha}

\date{}
\maketitle

\begin{abstract}
In this paper, as part of a project initiated by A. Mallios
consisting of exploring new horizons for \textit{Abstract
Differential Geometry} ($\grave{a}$ la Mallios), \cite{mallios1997,
mallios, malliosvolume2, modern}, such as those related to the
\textit{classical symplectic geometry}, we show that results
pertaining to biorthogonality in pairings of vector spaces do hold
for biorthogonality in pairings of $\mathcal A$-modules. However,
for the \textit{dimension formula} the algebra sheaf $\mathcal A$ is
assumed to be a PID. The dimension formula relates the rank of an
$\mathcal A$-morphism and the dimension of the kernel (sheaf) of the
same $\mathcal A$-morphism with the dimension of the source free
$\mathcal A$-module of the $\mathcal A$-morphism concerned. Also, in
order to obtain an analog of the Witt's hyperbolic decomposition
theorem, $\mathcal A$ is assumed to be a PID while topological
spaces on which $\mathcal A$-modules are defined are assumed
\textit{connected}.
\end{abstract}

\maketitle

\section{Introduction}

In this paper, we discuss \textit{biorthogonality in pairings of
$\mathcal A$-modules} and the \textit{Witt's hyperbolic
decomposition theorem for $\mathcal A$-modules,} where the
\textit{algebra sheaf $\mathcal A$} in the second part of the work
is considered to be a \textit{PID,} that is for every open
$U\subseteq X$, the algebra $\mathcal{A}(U)$ is a PID-algebra; in
other words, given a free $\mathcal A$-module $\mathcal E$ and a
sub-$\mathcal A$-module $\mathcal{F}\subseteq \mathcal E$, one has
that $\mathcal F$ is \textit{section-wise free.} See
\cite{cartandieudonne}.

All the $\mathcal A$-modules in the paper are defined on a fixed
topological space $X$. For the purpose of the Witt's hyperbolic
decomposition theorem, the space $X$ is assumed to be
\textit{connected.} The connectedness of $X$ leads to some useful
simplification, such as: \textit{Locally free $\mathcal A$-modules
of varying finite rank (Definition \ref{definition2}) with $\mathcal
A$ a PID algebra sheaf are locally free $\mathcal A$-modules of
finite rank, i.e. vector sheaves.}

For the sake of easy referencing, we recall some notions, which may
be found in our recent papers: \cite{malliosntumbaqm1},
\cite{darboux}, \cite{malliosntumbaqm2} and \cite{cartandieudonne}.
Let $\mathcal E$ and $\mathcal F$ be $\mathcal A$-modules and $\phi:
\mathcal{E}\oplus \mathcal{F}\longrightarrow \mathcal A$ an
$\mathcal A$-bilinear morphism (or $\mathcal A$-bilinear form). The,
we say that the triple $(\mathcal{F}, \mathcal{E}; \phi)\equiv
(\mathcal{F}, \mathcal{E}; \mathcal{A})\equiv ((\mathcal{F},
\mathcal{E}; \phi); \mathcal{A})$ constitutes a \textit{pairing of
$\mathcal A$-modules} or an \textit{$\mathcal A$-pairing,} in short.
Note the places of $\mathcal F$ in both the direct sum
$\mathcal{E}\oplus \mathcal F$ and the triple $(\mathcal{F},
\mathcal{E}; \phi).$ This placing is in keeping with the classical
case as is in Artin's book \cite{artin}. The sub-$\mathcal A$-module
$\mathcal{F}^{\top^\phi}$ of $\mathcal E$ such that, for every open
subset $U$ of $X$, $\mathcal{F}^{\top^\phi}(U)$ consists of all
$s\in \mathcal{E}(U)$ with $\phi_V(s|_V, \mathcal{F}(V))=0$ for any
open $V\subseteq U$, is called the \textit{left kernel} of
$(\mathcal{F}, \mathcal{E}; \phi)$. In a similar way, one defines
the \textit{right kernel} of $(\mathcal{F}, \mathcal{E}; \phi)$ to
be the sub-$\mathcal A$-module $\mathcal{E}^{\perp_\phi}$ of
$\mathcal F$ such that, for any open subset $U$ of $X$,
$\mathcal{E}^{\perp_\phi}(U)$ consists of all sections
$s\mathcal{F}(U)$ such that $\phi_V(\mathcal{E}(V), s|_V)=0$ for
every open $V\subseteq U.$ If $(\mathcal{F}, \mathcal{E}; \phi)$ is
a \textit{pairing of free $\mathcal A$-modules} (or a \textit{free
$\mathcal A$-pairing} for short), for every open subset $U$ of $X$,
\[
\mathcal{F}^{\top^\phi}(U)= \mathcal{F}(U)^{\top^\phi}:= \{s\in
\mathcal{E}(U):\ \phi_U(s, \mathcal{F}(U))=0\};
\]
similarly,
\[
\mathcal{E}^{\perp_\phi}(U)= \mathcal{E}(U)^{\perp_\phi}:= \{s\in
\mathcal{F}(U):\ \phi_U(\mathcal{E}(U), s)=0\}.
\]

Now, let $(\mathcal{E}, \mathcal{E}; \phi)\equiv (\mathcal{E},
\phi)$ be a self $\mathcal{A}$-pairing such that if $r, s\in
\mathcal{E}(U),$ where $U$ is an open subset of $X$, then $\phi_U(r,
s)=0$ if and only if $\phi_U(s, r)=0.$ The left kernel
$\mathcal{E}^{\top^\phi}$ is the same as the right kernel
$\mathcal{E}^{\perp_\phi}.$ In this case, we say that the $\mathcal
A$-bilinear form $\phi$ is \textit{orthosymmetric} and call
$\mathcal{E}^{\perp_\phi}(= \mathcal{E}^{\top^\phi})$ the
\textit{radical sheaf} of $\mathcal E$, and denote it by
rad$_\phi\mathcal{E}\equiv \mbox{rad}\ \mathcal E$. A self $\mathcal
A$-pairing such that $(\mathcal{E}, \phi)$ such that $\mbox{rad}\
\mathcal{E}\neq 0$ (resp. rad $\mathcal{E}=0$) is called
\textit{isotropic} (resp. \textit{non-isotropic}); $\mathcal E$ is
\textit{totally isotropic} if $\phi$ is identically zero, i.e.,
$\phi_U(r, s)=0$ for all sections $r, s\in \mathcal{E}(U)$, with $U$
any open subset of $X$. For any open $U\subseteq X$, a
\textit{non-zero section} $s\in \mathcal{E}(U)$ is called
\textit{isotropic} if $\phi_U(s, s)=0.$ The \textit{$\mathcal
A$-radical} of a sub-$\mathcal A$-module $\mathcal F$ of $\mathcal
E$ is defined as rad $\mathcal{F}:= \mathcal{F}\cap
\mathcal{F}^{\perp_\phi}= \mathcal{F}\cap \mathcal{F}^{\top^\phi}.$
If $(\mathcal{F}, \mathcal{E}; \phi)$ is a free $\mathcal
A$-pairing, then for every open subset $U\subseteq X$,
\[\begin{array}{lll}
(\mbox{rad}\ \mathcal{E})(U)= \mbox{rad}\ \mathcal{E}(U) &
\mbox{and} & (\mbox{rad}\ \mathcal{F})(U)= \mbox{rad}\
\mathcal{F}(U),\end{array}
\]
where $\mbox{rad}\ \mathcal{E}(U)= \mathcal{E}(U)\cap
\mathcal{E}(U)^{\perp_\phi}$ and $\mbox{rad}\ \mathcal{F}(U)=
\mathcal{F}(U)\cap \mathcal{F}(U)^{\top^\phi}.$ Given an $\mathcal
A$-pairing $(\mathcal{E}, \mathcal{E}; \phi)$ with $\phi$ a
symmetric or antisymmetric $\mathcal A$-bilinear morphism,
sub-$\mathcal A$-modules $\mathcal{E}_1$ and $\mathcal{E}_2$ are
said to be \textit{mutually orthogonal} if for every open subset $U$
of $X$, $\phi_U(r, s)=0,$ for all $r\in \mathcal{E}_1(U)$ and $s\in
\mathcal{E}_2(U)$. If $\mathcal{E}= \oplus_{i\in I}\mathcal{E}_i,$
where the $\mathcal{E}_i$ are pairwise orthogonal sub-$\mathcal
A$-modules of $\mathcal E$, we say that $\mathcal E$ is the direct
orthogonal sum of the $\mathcal{E}_i$, and write $\mathcal{E}:=
\mathcal{E}_1\bot\cdots \bot \mathcal{E}_i\bot \cdots .$

This paper grew out of our earlier efforts to understand the
conditions defining \textit{convenient $\mathcal A$-modules} (cf.
\cite{malliosntumba2} and \cite{cartandieudonne}). For the sake of
self-containedness, we recall a convenient $\mathcal A$-module is a
self $\mathcal A$-pairing $\mathcal{E}, \phi)$, where $\mathcal E$
is a free $\mathcal A$-module of finite rank and $\phi$ an
orthosymmetric $\mathcal A$-bilinear morphism, such that the
following conditions are satisfied: \textit{$(1)$ If $\mathcal F$ is
a free sub-$\mathcal A$-module of $\mathcal E$, then
$\mathcal{F}^\perp\equiv \mathcal{F}^{\perp_\phi}=
\mathcal{F}^{\top^\phi}\equiv \mathcal{F}^\top$ is a free
sub-$\mathcal A$-module of $\mathcal E$; $(2)$ Every free
sub-$\mathcal A$-module $\mathcal F$ of $\mathcal E$ is orthogonally
reflexive, i.e., $\mathcal{F}^{\perp\top}= \mathcal F$; $(3)$ The
intersection of any two free sub-$\mathcal A$-modules of $\mathcal
E$ is a free sub-$\mathcal A$-module.} While symplectic vector
spaces satisfy the above conditions, we have not come up yet with an
example other than vector spaces. However, we show in the present
account that if the space $X$ is connected and the algebra sheaf
$\mathcal A$ a PID, then every \textit{symplectic orthogonally
convenient $\mathcal A$-pairing} satisfies the conditions above.
Every free $\mathcal A$-module $\mathcal E$ has a \textit{naturally
associated orthogonally convenient $\mathcal A$-pairing}: the
$\mathcal A$-pairing $(\mathcal{E}^\ast, \mathcal{E}; \nu)$, where
\[
\nu_U(s, \psi):= \psi_U(s)\in \mathcal{A}(U)
\]
for every open subset $U\subseteq X$ and sections $s\in
\mathcal{E}(U)$, $\psi\equiv (\psi_V)_{U\supseteq V,\ open}\in
\mathcal{E}^\ast(U).$ The $\mathcal A$-pairing $(\mathcal{E}^\ast,
\mathcal{E}; \nu)$ is called the \textit{canonical
$\mathcal{A}$-pairing of $\mathcal E$ and $\mathcal{E}^\ast$.}

Condition $(2)$ of convenient $\mathcal A$-modules has prompted us
to study extensively biorthogonality in pairings of $\mathcal
A$-modules. One may refer to \cite{artin}, \cite{chambadal},
\cite{deheuvels}, \cite{crumeyrolle}, \cite{lang} and \cite{omeara}
for biorthogonality in pairings of vector spaces.

\textit{Notation.} We assume throughout the paper, unless otherwise
mentioned, that the pair $(X, \mathcal{A})$ is an
\textit{algebraized space} (\cite[p. 96]{mallios}), where $\mathcal
A$ is a unital $\mathbb C$-algebra sheaf such that \textit{every
nowhere-zero section of $\mathcal A$ is invertible.} Furthermore,
all free $\mathcal A$-modules are considered to be
\textit{torsion-free,} that is, for eny open subset $U\subseteq X$
and nowhere-zero section $s\in \mathcal{E}(U)$, if $as=0$, where
$a\in \mathcal{A}(U),$ then $a=0.$ Next, in the course of the paper,
the notation $s\in \mathcal{E}(U)$ signifies that $s$ is a section
of an $\mathcal A$-module $\mathcal E$ over an open subset
$U\subseteq X$. Finally, left and right kernels in a canonical
$\mathcal A$-pairing $\mathcal{E}^\ast, \mathcal{E}; \nu)$ are
simply denoted using superscripts $\perp$ and $\top$ instead of the
more formal ones $\perp_\nu$ and $\top^\nu$.

\section{Universal property of quotient $\mathcal A$-modules}

This section contains proofs of the basic results on biorthogonality
in canonical parings of $\mathcal A$-modules, namely Proposition
\ref{proposition2} and Theorems \ref{theorem4} and \ref{theorem6}.

\begin{theorem}
Let $\mathcal E$, $\mathcal F$ and $\mathcal G$ be $\mathcal
A$-modules.

$1.$ Let $\phi\in \emph{H}om_\mathcal{A}(\mathcal{E}, \mathcal{F})$
be a surjective $\mathcal A$-morphism. Then, if $\psi\in
\emph{H}om_\mathcal{A}(\mathcal{E}, \mathcal{G})$ such that $\ker
\phi\subseteq \ker \psi,$ there exists a unique $\theta\in
\emph{H}om_\mathcal{A}(\mathcal{F}, \mathcal{G})$ such that the
diagram
\[
\xymatrix{\mathcal{E}\ar[r]^\phi\ar[dr]_\psi &
\mathcal{F}\ar@{.>}[d]^\theta\\ & \mathcal{G}}
\]
commutes. In other words, the mapping $\theta\mapsto \theta\circ
\phi$ is an $\mathcal A$-isomorphism from
$\emph{H}om_\mathcal{A}(\mathcal{F}, \mathcal{G})$ onto the
sub-$\mathcal A$-module of $\emph{H}om_\mathcal{A}(\mathcal{E},
\mathcal{G})$ consisting of $\mathcal A$-morphisms whose kernel
contains $\ker \phi.$

$2.$ Let $\phi\in \emph{H}om_\mathcal{A}(\mathcal{F}, \mathcal{G})$
be an injective $\mathcal A$-morphism. Then, if $\psi\in
\emph{H}om_\mathcal{A}(\mathcal{E}, \mathcal{G})$ such that
$\emph{\mbox{Im}} \psi\subseteq \emph{\mbox{Im}} \phi$, there exists
a unique $\theta\in \emph{H}om_\mathcal{A}(\mathcal{E},
\mathcal{F})$ making the diagram
\[
\xymatrix{\mathcal{E}\ar@{.>}[d]_\theta\ar[rd]^\psi & \\
\mathcal{F}\ar[r]_\phi & \mathcal{G}}
\]
commute. More precisely, the mapping $\theta\mapsto \phi\circ
\theta$ is an $\mathcal A$-isomorphism from
$\emph{H}om_\mathcal{A}(\mathcal{E}, \mathcal{F})$ onto the
sub-$\mathcal A$-module of $\emph{H}om_\mathcal{A}(\mathcal{E},
\mathcal{G})$ consisting of $\mathcal A$-morphism whose image is
contained in $\emph{\mbox{Im}} \phi$. \label{theorem1}
\end{theorem}

\begin{proof}
\textit{Assertion $1.$} \textbf{Uniqueness.} Let $\theta_1$,
$\theta_2\in \emph{H}om_\mathcal{A}(\mathcal{F}, \mathcal{G})$ be
such that $\psi= \theta_1\circ \phi$ and $\psi= \theta_2\circ \phi.$
Fix an open subset $U$ in $X$; since $\phi_U$ is surjective, the
equation $\theta_{1, U}\circ \phi_U= \theta_{2, U}\circ \phi_U$
implies that $\theta_{1, U}= \theta_{2, U}.$ Thus, $\theta_1=
\theta_2.$

\textbf{Existence.} Fix an open subset $U$ in $X$ and consider an
element (section) $t\in \mathcal{F}(U)$. Since $\phi_U$ is
surjective, there exists an element $s\in \mathcal{E}(U)$ such
that $t=\phi_U(s).$ Now, suppose there exists a $r\in
\mathcal{F}(U)$ with $u\in \ker \psi_U$ and $v\notin \ker \psi_U$
as its pre-images by $\phi_U$, i.e.
\[
\phi_U(v)= r= \phi_U(u)
\]
with $u\in \ker \psi_U$ and $v\notin \ker \psi_U.$ Since $\phi_U$
is linear, $\phi_U(v-u)=0$; so $v-u\in \ker \phi_U\subseteq \ker
\psi_U.$ But $u\in \ker \psi_U,$ so $v\in \ker \psi_U,$ which
yields a \textit{contradiction}. We conclude that such a situation
cannot occur. Furthermore, the element $\psi_U(s)$ does only
depend on $t$. Let $\theta_U$ be the $\mathcal{A}(U)$-morphism
sending $\mathcal{F}(U)$ into $\mathcal{G}(U)$ and such that
\[
\theta_U(t)= \psi_U(s);
\]
that
\[
\psi_U= \theta_U\circ \phi_U
\]
is clear.

Next, let us consider the \textit{complete presheaves of sections}
of $\mathcal E$, $\mathcal F$ and $\mathcal G$, respectively, viz.
\[\begin{array}{lll}\Gamma(\mathcal{E})\equiv (\Gamma(U,
\mathcal{E}), \alpha^U_V), & \Gamma(\mathcal{F})\equiv (\Gamma(U,
\mathcal{F}), \beta^U_V), & \Gamma(\mathcal{G})\equiv (\Gamma(U,
\mathcal{G}), \delta^U_V).\end{array}
\]
Given open subsets $U$ and $V$ of $X$ such that $V\subseteq U$,
since $\psi\in \emph{H}om_\mathcal{A}(\mathcal{E}, \mathcal{G})$,
one has
\begin{equation}\label{eq1}
\psi_V\circ \alpha^U_V= \delta^U_V\circ \psi_U.
\end{equation}
But $\psi_U= \theta_U\circ \phi_U$ and $\psi_V= \theta_V\circ
\phi_V,$ therefore, (\ref{eq1}) becomes
\[
\theta_V\circ \phi_V\circ \alpha^U_V= \delta^U_V\circ
\theta_U\circ \phi_U
\]
or
\begin{equation}\label{eq2}
\theta_V\circ \beta^U_V\circ \phi_U= \delta^U_V\circ \theta_U\circ
\phi_U.
\end{equation}
Since $\phi_U$ is surjective, it follows from (\ref{eq2}) that
\[
\theta_V\circ \beta^U_V= \delta^U_V\circ \theta_U,
\]
which means that $\theta\equiv (\theta_U)_{X\supseteq U,\ open}$
is an $\mathcal A$-morphism of $\mathcal F$ into $\mathcal G$ such
that
\[
\psi= \theta\circ \phi,
\]
as required.

\textit{Assertion $2.$} \textbf{Uniqueness.} Let $\theta_1,\
\theta_2\in \emph{H}om_\mathcal{A}(\mathcal{E}, \mathcal{F})$ be
such that $\psi= \phi\circ \theta_1$ and $\psi= \phi\circ \theta_2$.
As $\phi$ is injective, the relation
\[
\phi\circ \theta_1= \phi\circ \theta_2
\]
implies that $\theta_1= \theta_2,$ so uniqueness is obtained.

\textbf{Existence.} Fix an open subset $U$ in $X$ and consider an
element $s\in \mathcal{E}(U);$ since $\mbox{Im} \psi\subseteq
\mbox{Im} \phi,$ there exists a $t\in \mathcal{F}(U)$ such that
\begin{equation}\label{eq3}
\phi_U(t)= \psi_U(s).
\end{equation}
But $\phi_U$ is injective, therefore such an element $t$ is
unique. Now, let $\theta_U$ be the mapping of $\mathcal{E}(U)$
into $\mathcal{F}(U)$ sending an element $s\in \mathcal{E}(U)$ to
an element $t\in \mathcal{F}(U)$ such that (\ref{eq3}) is
satisfied. It is immediate that $\theta_U$ is
$\mathcal{A}(U)$-linear, and one has
\[
\psi_U= \phi_U\circ \theta_U.
\]

Finally, let $\Gamma(\mathcal{E})\equiv (\Gamma(U, \mathcal{E}),
\alpha^U_V),\  \Gamma(\mathcal{F})\equiv (\Gamma(U, \mathcal{F}),
\beta^U_V), \ \Gamma(\mathcal{G})\equiv (\Gamma(U, \mathcal{G}),
\delta^U_V)$ be as above the complete presheaves of sections of
$\mathcal E$, $\mathcal F$ and $\mathcal G$, respectively. Given
open subsets $U$ and $V$ of $X$ such that $V\subseteq U,$ since
$\psi\in \emph{H}om_\mathcal{A}(\mathcal{E}, \mathcal{G}),$ one has
\begin{equation}\label{eq4}
\psi_V\circ \alpha^U_V= \delta^U_V\circ \psi_U.
\end{equation}
But $\psi_U= \phi_U\circ \theta_U$ and $\psi_V= \phi_V\circ
\theta_V$, therefore, we deduce from (\ref{eq4}) that
\[
\phi_V\circ \theta_V\circ \alpha^U_V= \delta^U_V\circ \phi_U\circ
\theta_U
\]
or
\begin{equation}\label{eq5}
\phi_V\circ \theta_V\circ \alpha^U_V= \phi_V\circ \beta^U_V\circ
\theta_U.
\end{equation}
Since $\phi_V$ is injective, it is clear from (\ref{eq5}) that
\[
\theta_V\circ \alpha^U_V= \beta^U_V\circ \theta_U,
\]
which is to say that $\theta\equiv (\theta_U)_{X\supseteq U,\
open}$ is an $\mathcal A$-morphism of $\mathcal E$ into $\mathcal
F$ such that
\[
\psi= \phi\circ \theta,
\]
and the proof is complete.
\end{proof}

The \textit{universal property of quotient $\mathcal A$-modules} is
then obtained as a corollary of Theorem \ref{theorem1}. More
precisely, one has

\begin{corollary}$($\textbf{Universal property of quotient $\mathcal
A$-modules}$)$ Let $\mathcal E$ be an $\mathcal A$-module,
$\mathcal{E}'$ a sub-$\mathcal A$-module of $\mathcal E$, and $\phi$
the canonical $\mathcal A$-morphism of $\mathcal E$ onto
$\mathcal{E}/\mathcal{E}'$. The pair $(\mathcal{E}/\mathcal{E}',
\phi)$ satisfies the following universal property:

Given any pair $(\mathcal{F}, \psi)$ consisting of an $\mathcal
A$-module $\mathcal F$ and an $\mathcal A$-morphism $\psi\in
\emph{H}om_\mathcal{A}(\mathcal{E}, \mathcal{F})$ such that
$\mathcal{E}'\subseteq \ker \psi,$ there exists a unique $\mathcal
A$-morphism $\widetilde{\psi}\in
\emph{H}om_\mathcal{A}(\mathcal{E}/\mathcal{E}', \mathcal{F})$ such
that the diagram
\[
\xymatrix{\mathcal{E}\ar[r]^\phi\ar[dr]_\psi &
\mathcal{E}/\mathcal{E}'\ar@{.>}[d]^{\widetilde{\psi}}\\ &
\mathcal{F}}
\]
commutes, i.e.
\[
\psi= \widetilde{\psi}\circ \phi.
\]

The kernel of $\widetilde{\psi}$ equals the image by $\phi$ of the
kernel of $\psi$, and the image of $\widetilde{\psi}$ equals the
image of $\psi.$

The mapping
\[
\theta\mapsto \theta\circ \phi
\]
is an $\mathcal{A}$-isomorphism of the $\mathcal{A}$-module
$\emph{H}om_\mathcal{A}(\mathcal{E}/\mathcal{E}', \mathcal{F})$ onto
the sub-$\mathcal A$-module of $\emph{H}om_\mathcal{A}(\mathcal{E},
\mathcal{F})$ consisting of $\mathcal A$-morphisms of $\mathcal{E}$
into $\mathcal F$ whose kernel contains
$\mathcal{E}'$.\label{corollary2}
\end{corollary}

\begin{proof}
Apply assertion $1$ of Theorem \ref{theorem1}.
\end{proof}

Similarly to the classical case (cf. \cite[p. 15, Corollary
1]{chambadal}), we also have the following corollary, the proof of
which is an easy exercise and is, for that reason, omitted.

\begin{corollary}\label{corollary1}
Let $\mathcal E$ and $\mathcal F$ be $\mathcal A$-modules and
$\phi\in \emph{H}om_\mathcal{A}(\mathcal{E}, \mathcal{F}).$ Then,
\begin{enumerate}
\item [{$(1)$}] $\mathcal{E}/\ker \phi= \emph{\mbox{Im}} \phi$
within an $\mathcal A$-isomorphism. \item [{$(2)$}] Given a
sub-$\mathcal A$-module $\mathcal{F}'$ of $\mathcal F$,
$\mathcal{E}'\equiv \phi^{-1}(\mathcal{F}')$ is a sub-$\mathcal
A$-module of $\mathcal E$ containing $\ker \phi;$ moreover,
$\mathcal{F}'= \phi(\mathcal{E}')$ if $\phi$ is surjective.
\item[{$(3)$}] Conversely, if $\mathcal{E}'$ is a sub-$\mathcal
A$-module of $\mathcal E$ containing $\ker \phi$, then
$\mathcal{F}'\equiv \emph{\mbox{Im}} \mathcal{E}'$ is a
sub-$\mathcal A$-module of $\mathcal F$ such that $\mathcal{E}'=
\phi^{-1}(\mathcal{F}').$
\end{enumerate}
\end{corollary}

As a further application of the universal property of quotient
$\mathcal A$-modules, we have

\begin{corollary}\label{corollary3}
Let $\mathcal E$ be a free $\mathcal A$-module, and $\mathcal{E}_1$
a free sub-$\mathcal A$-module of $\mathcal E$. Then, the $\mathcal
A$-morphism $\phi\equiv (\phi_U)_{X\supseteq U,\ open}\in
\emph{H}om_{\mathcal{A}}(\mathcal{E}^\ast, \mathcal{E}^\ast_1)$ such
that every $\phi_U$ maps an element $(\psi_V)_{U\supseteq V,\ open}$
of $\Gamma(\mathcal{E}^\ast)(U)\equiv \mathcal{E}^\ast(U)$ onto its
restriction $({\psi_V}|_{\mathcal{E}_1(V)})_{U\supseteq V,\ open}\in
\mathcal{E}_1^\ast(U)$ is surjective, and has
$\mathcal{E}_1^\perp\subseteq \mathcal{E}^\ast$ as its kernel.
Moreover,
\[
\mathcal{E}^\ast/\mathcal{E}_1^\perp= \mathcal{E}_1^\ast
\]
within an $\mathcal A$-isomorphism.
\end{corollary}

\begin{proof}
That $\ker \phi= \mathcal{E}_1^\perp$ is clear. Now, let
$\mathcal{E}_2$ be a free sub-$\mathcal A$-module of $\mathcal E$
complementing $\mathcal{E}_1$. It follows (cf. \cite[p. 137,
relation (6.21)]{mallios} that
\[
\mathcal{E}^\ast= \mathcal{E}_1^\ast\oplus \mathcal{E}_2^\ast,
\]
so that if $U$ is open in $X$ and
\[\begin{array}{lll}
\psi\equiv (\psi_V)_{U\supseteq V,\ open}\in \mathcal{E}_1^\ast(U) &
\mbox{and} & \theta\equiv (\theta_V)_{U\supseteq V,\ open}\in
\mathcal{E}_2^\ast(U), \end{array}
\] then
\[
\Omega\equiv \psi+ \theta\in \mathcal{E}^\ast(U).
\]
If $V$ is open in $U$ and $s\in \mathcal{E}(V),$ so $s$ is uniquely
written as $s= r+t$ where $r\in \mathcal{E}_1(V)$ and $t\in
\mathcal{E}_2(V),$ then
\[
\Omega_V(s)= \psi_V(s)+ \theta_V(t).
\]
Consequently,
\[
\phi_U(\Omega)= ({\Omega_V}|_{\mathcal{E}_1(V)})_{U\supseteq V,\
open}= \psi;
\]
thus $\phi_U$ is surjective. Hence, applying Corollary
\ref{corollary1} $(1)$, we obtain an $\mathcal A$-isomorphism
\[
\mathcal{E}^\ast/\mathcal{E}_1^\perp\simeq \mathcal{E}_1^\ast.
\]
\end{proof}

Now, let us introduce the notion of \textit{$\mathcal
A$-projection.}

\begin{definition}
\emph{Let $\mathcal E$ be an $\mathcal A$-module, $\mathcal F$ and
$\mathcal G$ two \textit{supplementary sub-$\mathcal A$-modules} of
$\mathcal E$. (Thus, for every open subset $U\subseteq X$, every
section $s\in \mathcal{E}(U)$ can be uniquely written as $s= r+t$,
where $r\in \mathcal{F}(U)$ and $t\in \mathcal{G}(U).$) The
$\mathcal{A}$-\textit{endomorphism}
\[
\pi^\mathcal{F}\equiv (\pi^\mathcal{F}_U)_{X\supseteq U,\ open}\in
\emph{H}om_{\mathcal{A}}(\mathcal{E}, \mathcal{E}):=
\mathcal{E}nd_{\mathcal{A}}(\mathcal{E})
\]
such that, for any section $s\in \mathcal{E}(U)\equiv
\Gamma(\mathcal{E})(U):= \Gamma(U, \mathcal{E}),$
\[
\pi^\mathcal{F}_U(s)\equiv \pi^\mathcal{F}_U(r+ t):= r,
\]
where $s= r+t$ with $r\in \mathcal{F}(U)$ and $t\in \mathcal{G}(U),$
is called the \textbf{$\mathcal{A}$-projection onto} $\mathcal F$
(\textbf{parallel to} $\mathcal G$). In a similar way, one define
the \textbf{$\mathcal A$-projection onto} $\mathcal G$
(\textbf{parallel to} $\mathcal F$).}
\end{definition}

\begin{proposition}\label{proposition2}
Let $\mathcal E$ be a free $\mathcal A$-module, $\mathcal{E}_1$ and
$\mathcal{E}_2$ two free sub-$\mathcal A$-modules of $\mathcal E$
the direct sum of which is $\mathcal A$-isomorphic to $\mathcal E$,
$\pi_1\equiv \pi^{\mathcal{E}_1},$ $\pi_2\equiv \pi^{\mathcal{E}_2}$
the corresponding $\mathcal A$-projections. Then,
\[
\mathcal{E}^\ast= \mathcal{E}_1^\perp\oplus \mathcal{E}_2^\perp,
\]
and the $\mathcal A$-projections $\pi'_1\equiv
\pi^{\mathcal{E}_1^\perp},$ $\pi'_2\equiv \pi^{\mathcal{E}_2^\perp}$
associated with this direct decomposition are given by setting
\[\begin{array}{lll}
\pi'_{1, U}(\alpha):= (\alpha_V\circ \pi_{2, V})_{U\supseteq V,\
open} & \mbox{and} & \pi'_{2, U}(\alpha):= (\alpha_V\circ \pi_{1,
V})_{U\supseteq V,\ open} \end{array}
\]
for any $\alpha\equiv (\alpha_V)_{U\supseteq V,\ open}\in
\mathcal{E}^\ast(U).$
\end{proposition}

The proof of Proposition \ref{proposition2} requires some part of
\cite[p. 404, Theorem 2.2]{malliosntumbaqm1}, which we restate here
for easy referencing.

\begin{theorem}\label{theorem2}
Let $(\mathcal{E}^\ast, \mathcal{E}; \mathcal{A})$ be the canonical
free $\mathcal A$-pairing determined by $\mathcal E$. Then, for any
open subset $U\subseteq X$,
\[
\dim \mathcal{E}^\ast(U)= \dim \mathcal{E}(U).
\]
If $\phi\in \mathcal{E}^\ast(U)$ and $\phi_U(s)=0$ for all $s\in
\mathcal{E}(U),$ then $\phi=0;$ on the other hand, if $\phi(s)=0$
for all $\phi\in \mathcal{E}^\ast(U),$ then $s=0.$
\end{theorem}

Now, let us get to the proof of Proposition \ref{proposition2}.

\begin{proof}(\textbf{Proposition \ref{proposition2}})
Fix an open set $U$ in $X$. That $(\alpha_V\circ \pi_{2,
V})_{U\supseteq V,\ open}$ and $(\alpha_V\circ \pi_{1,
V})_{U\supseteq V,\ open}$ belong to $\mathcal{E}_1^\perp(U)$ and
$\mathcal{E}_2^\perp(U),$ respectively, is obvious. For any open
$V\subseteq U,$ the relation
\[
\alpha_V= \alpha_V\circ \pi_{1, V}+ \alpha_V\circ \pi_{2, V}
\]
shows that
\[
\mathcal{E}^\ast(U)= \mathcal{E}_1^\perp(U)+ \mathcal{E}_2^\perp(U).
\]

Finally, suppose that there exists $\beta\equiv
(\beta_V)_{U\supseteq V,\ open}$ in $\mathcal{E}_1^\perp(U)\cap
\mathcal{E}_2^\perp(U);$ since $\beta_V(s)=0$ for any open
$V\subseteq U$ and any $s\in \mathcal{E}(V)= \mathcal{E}_1(V)\oplus
\mathcal{E}_2(V),$ it follows that $\beta=0$ (cf. Theorem
\ref{theorem2}). Thus,
\[
\mathcal{E}^\ast(U)= \mathcal{E}_1^\perp(U)\oplus
\mathcal{E}_2^\perp(U)
\]
and hence
\[
\mathcal{E}^\ast= \mathcal{E}_1^\perp\oplus \mathcal{E}_2^\perp
\]
as claimed.
\end{proof}

From \cite{malliosntumbaqm1}, we also recall the following result a
particular case of which will be needed below.

\begin{theorem}\label{theorem3}
Let $(\mathcal{F}, \mathcal{E}; \mathcal{A})$ be an $\mathcal
A$-pairing such that the right $\mathcal A$-kernel, i.e.
$\mathcal{E}^\perp$, is identically $0$. Moreover, let
$\mathcal{E}_0$ and $\mathcal{F}_0$ be sub-$\mathcal A$-modules of
$\mathcal E$ and $\mathcal F$, respectively. Then, there exist
natural $\mathcal A$-isomorphisms \textbf{into}:
\[\begin{array}{lll}
\mathcal{E}/\mathcal{F}_0^\perp\longrightarrow \mathcal{F}_0^\ast &
\mbox{and} & \mathcal{E}_0^\perp\longrightarrow
(\mathcal{E}/\mathcal{E}_0)^\ast. \end{array}
\]
\end{theorem}

An  interesting result may be derived from Theorem \ref{theorem3},
viz.:

\begin{theorem}\label{theorem4}
Let $\mathcal E$ be a free $\mathcal A$-module, $\mathcal{E}_1$ a
free sub-$\mathcal A$-module of $\mathcal E$, and $\phi$ the
canonical $\mathcal A$-morphism of $\mathcal E$ onto $($the free
sub-$\mathcal A$-module$)$ $\mathcal{E}/\mathcal{E}_1$. The
$\mathcal A$-morphism
\[
\Lambda\equiv (\Lambda_U)_{X\supseteq U,\ open}:
(\mathcal{E}/\mathcal{E}_1)^\ast\longrightarrow \mathcal{E}^\ast
\]
such that, given any open subset $U\subseteq X$ and a section
$\psi\equiv (\psi_V)_{U\supseteq V,\ open}\in
(\mathcal{E}/\mathcal{E}_1)^\ast(U):=
\emph{\mbox{H}}om_{\mathcal{A}|_U}((\mathcal{E}/\mathcal{E}_1)|_U,
\mathcal{A}|_U),$
\[
\Lambda_U(\psi):= (\psi_V\circ \phi_V)_{U\supseteq V,\ open}
\]
is an $\mathcal A$-isomorphism of $(\mathcal{E}/\mathcal{E}_1)^\ast$
onto $\mathcal{E}_1^\perp,$ where $\mathcal{E}_1^\perp$ is the
$\mathcal A$-orthogonal of $\mathcal{E}_1$ in the canonical
$\mathcal A$-pairing $(\mathcal{E}^\ast, \mathcal{E}; \mathcal{A}).$
\end{theorem}

\begin{proof}
It is clear that $\Lambda$ is indeed an $\mathcal A$-morphism. Now,
let us fix an open set $U$ in $X$ and let us consider a section
$\psi\equiv (\psi_V)_{U\supseteq V,\ open}\in
(\mathcal{E}/\mathcal{E}_1)^\ast(U).$ Then, $\Lambda_U(\psi)=0$ if
for any open $V$ in $U$ and $s\in \mathcal{E}(V),$
\[
\Lambda_U(\psi)(s)=0.
\]
But
\[
\Lambda_U(\psi)(s)= (\psi_V\circ \phi_V)(s)= \psi_V(\phi_V(s))=0,
\]
therefore, by Theorem \ref{theorem2},
\[
\psi_V=0.
\]
It follows that
\[
\ker \Lambda_U=0,
\]
and consequently
\[
\ker \Lambda=0;
\]
in other words, $\Lambda$ is injective.

Next, for every $\psi\equiv (\psi_V)_{U\supseteq V,\ open}\in
(\mathcal{E}/\mathcal{E}_1)^\ast(U),$ where the open set $U$ is
fixed in $X$,
\[
\Lambda_U(\psi)(s)= (\psi_V\circ \phi_V)(s)=0,
\]
where $s$ is any element in $\mathcal{E}_1(V);$ that is
\[
\Lambda_U(\psi)\in \mathcal{E}_1^\perp(U),
\]
from which we deduce that
\[
\mbox{Im} \Lambda\subseteq \mathcal{E}_1^\perp.
\]

Finally, still under the assumption that $U$ is an open set fixed in
$X$, let us consider, for every open $V\subseteq U,$ the following
commutative diagram
\[
\xymatrix{\mathcal{E}(V)\ar[r]^{\phi_V}\ar[dr]_{\psi_V\circ \phi_V}
& (\mathcal{E}/\mathcal{E}_1)(V)\ar@{.>}[d]^{\psi_V}\\ &
\mathcal{A}(V)}.
\]
The \textit{universal property of quotient $\mathcal A$-modules}
(cf. Corollary \ref{corollary2}) shows that, given an element
$\sigma_V\in \mbox{H}om_{\mathcal{A}(V)}(\mathcal{E}(V),
\mathcal{A}(V))$ such that $\ker \phi_V\subseteq \ker \sigma_V$,
i.e., $\sigma_V(\mathcal{E}_1(V))= 0$, there is a unique $\psi_V\in
\mbox{H}om_{\mathcal{A}(V)}((\mathcal{E}/\mathcal{E}_1)(V),
\mathcal{A}(V))$ such that
\[
\sigma_V= \psi_V\circ \phi_V.
\]
It is clear that the family $\sigma\equiv (\sigma_V)_{U\supseteq V,\
open}$ is an $\mathcal A$-morphism $\mathcal{E}|_U\longrightarrow
\mathcal{A}|_U$ satisfying the property that:
\[
\sigma= \psi\circ \phi.
\]
Thus, $\Lambda$ is surjective and the proof is finished.
\end{proof}

As a result, based essentially on everything above, we have

\begin{theorem}\label{theorem6}
Let $(\mathcal{E}^\ast, \mathcal{E}; \mathcal{A})$ be the canonical
free $\mathcal A$-pairing and $\mathcal{E}_1$ a free sub-$\mathcal
A$-module of $\mathcal E$. Then,
\begin{enumerate}
\item [{$(1)$}] $(\mathcal{E}_1^\perp)^\top= \mathcal{E}_1$ within
an $\mathcal A$-isomorphism.
\item [{$(2)$}] $\mathcal{E}_1$ has finite dimension if and only if
$\mathcal{E}_1^\perp$ has finite codimension in $\mathcal{E}^\ast$,
and then one has
\[
\dim \mathcal{E}_1=
\emph{\mbox{codim}}_{\mathcal{E}^\ast}\mathcal{E}_1^\perp.
\]
\item [{$(3)$}] $\mathcal{E}_1$ has finite codimension in $\mathcal
E$ if and only if $\mathcal{E}_1^\perp$ has finite dimension, and
\[
\emph{\mbox{codim}}_\mathcal{E}\mathcal{E}_1= \dim
\mathcal{E}_1^\perp.
\]
\end{enumerate}
\end{theorem}

\begin{proof}
\textit{Assertion $(1).$} Let $\mathcal{E}_2$ be a free
sub-$\mathcal A$-module of $\mathcal E$, complementing
$\mathcal{E}_1.$ By Proposition \ref{proposition2},
\[
\mathcal{E}^\ast\simeq \mathcal{E}_1^\perp\oplus
\mathcal{E}_2^\perp.
\]
We already know that $\mathcal{E}_1\subseteq
(\mathcal{E}_1^\perp)^\top.$ Now, consider a section $s\in
(\mathcal{E}_1^\perp)^\top(U);$ there exist $r\in \mathcal{E}_1(U)$
and $t\in \mathcal{E}_2(U)$ such that $s= r+ t.$ The section $t$ is
orthogonal to $\mathcal{E}^\perp_2(U)$, and since $r$ and $s$ are
orthogonal to $\mathcal{E}_1^\perp(U),$ we then have that $t$ is
orthogonal to $\mathcal{E}_1^\perp(U)\oplus
\mathcal{E}_2^\perp(U)\simeq \mathcal{E}^\ast(U).$ It follows from
Theorem \ref{theorem2} that $t=0$; thus
$(\mathcal{E}_1^\perp)^\top(U)\subseteq \mathcal{E}_1(U),$ and hence
$(\mathcal{E}_1^\perp)^\top\subseteq \mathcal{E}_1.$

\textit{Assertion $(2).$} Since $\mathcal{E}_1$ is free, it follows
that $\mathcal{E}_1^\ast\simeq \mathcal{E}_1$ (cf. \cite[p. 298,
(5.2)]{mallios}). Thus, $\mathcal{E}_1$ has finite dimension if and
only if $\mathcal{E}_1^\ast$ has finite dimension, and
\[
\dim \mathcal{E}_1^\ast= \dim \mathcal{E}_1.
\]
But, by Corollary \ref{corollary3},
$\mathcal{E}^\ast/\mathcal{E}_1^\perp$ is $\mathcal A$-isomorphic to
$\mathcal{E}_1^\ast,$ therefore
\[
\dim \mathcal{E}_1=
\mbox{codim}_{\mathcal{E}^\ast}\mathcal{E}_1^\perp.
\]

\textit{Assertion $(3).$} Let $\mathcal{E}_2$ be a free
sub-$\mathcal A$-module of $\mathcal E$ complementing
$\mathcal{E}_1$, that is $\mathcal{E}= \mathcal{E}_1\oplus
\mathcal{E}_2.$ But $\mathcal{E}/\mathcal{E}_1$ is $\mathcal
A$-isomorphic to $\mathcal{E}_2$ (cf. \cite{malliosntumbaqm2}),
therefore $\mathcal{E}/\mathcal{E}_1$ is free; consequently
$\mathcal{E}/\mathcal{E}_1$ has finite dimension if and only if
$(\mathcal{E}/\mathcal{E}_1)^\ast$ has finite dimension, and one has
\[
(\mathcal{E}/\mathcal{E}_1)^\ast\simeq \mathcal{E}/\mathcal{E}_1
\]
so that
\[
\mbox{codim}_\mathcal{E}\mathcal{E}_1= \dim
\mathcal{E}/\mathcal{E}_1= \dim (\mathcal{E}/\mathcal{E}_1)^\ast.
\]
But, by Theorem \ref{theorem4},
$(\mathcal{E}/\mathcal{E}_1)^\ast\simeq \mathcal{E}_1^\perp$ within
an $\mathcal A$-isomorphism, so the assertion is corroborated.
\end{proof}

\section{Biorthogonality in dual free $\mathcal A$-modules}

Let us introduce a set of notions we will be concerned with in the
sequel.

\begin{definition}\label{definition2}
An $\mathcal A$-module $\mathcal E$ is called a \textbf{locally free
$\mathcal A$-module of varying finite rank} if there exist an open
covering $\mathcal{U}\equiv (U_\alpha)_{\alpha\in I}$ of $X$ and
numbers $n(\alpha)\in \mathbb{N}$ for every open set $U_\alpha$ such
that
\[
\mathcal{E}|_{U_\alpha}= \mathcal{A}^{n(\alpha)}|_{U_\alpha}.
\]
The open covering $\mathcal{U}$ is called a \textbf{local frame.}
\end{definition}

It is clear that if the topological space $X$ is \textit{connected}
the \textit{local} (\textit{constant}) \textit{ranks} are
\textit{equal,} and thus locally free $\mathcal A$-modules of
varying finite rank on a connected topological space are
\textit{vector sheaves.} See \cite[p. 127, Definition 4.3]{mallios}
for vector sheaves. In particular, if the algebra sheaf $\mathcal A$
is a PID, we have the following result.

\begin{corollary}
Let $X$ be a connected topological space, $\mathcal A$ a PID algebra
sheaf on $X$ and $\mathcal E$ an $\mathcal A$-module on $X$. Then,
every sub-$\mathcal A$-module $\mathcal F$ of $\mathcal E$ is a
locally free sub-$\mathcal A$-module of finite rank, i.e. a vector
sheaf. Consequently, if $\mathcal{F}_1$ and $\mathcal{F}_2$ are
sub-$\mathcal A$-modules of $\mathcal E$, then $\mathcal{F}_1\cap
\mathcal{F}_2$ is a vector sheaf.
\end{corollary}

\begin{proof}
Obvious.
\end{proof}

\begin{example}
\emph{Consider a free $\mathcal A$-module $\mathcal E$, where
$\mathcal A$ is a PID-algebra sheaf. Then, every sub-$\mathcal
A$-module of $\mathcal E$ is a locally free $\mathcal A$-module of
varying finite rank.}
\end{example}

\begin{definition} Let
$(\mathcal{F}, \mathcal{E}; \phi)$ be a pairing of \textsf{free
$\mathcal A$-modules} $\mathcal E$ and $\mathcal F$.
\begin{enumerate}
\item [{$(i)$}] $(\mathcal{F}, \mathcal{E}; \phi)$  is called an
\textbf{orthogonally convenient $\mathcal A$-pairing} if for all
free sub-$\mathcal A$-modules $\mathcal{E}_0$ and $\mathcal{F}_0$ of
$\mathcal E$ and $\mathcal F$, respectively, their orthogonal
$\mathcal{E}_0^{\perp_\phi}$ and $\mathcal{F}_0^{\top^\phi}$ are
free sub-$\mathcal A$-modules of $\mathcal F$ and $\mathcal E$,
respectively.
\item [{$(ii)$}] $(\mathcal{F}, \mathcal{E}; \phi)$ is called a
\textbf{locally orthogonally convenient $\mathcal A$-pairing} if for
all locally free sub-$\mathcal A$-modules of varying finite rank
$\mathcal{E}_0$ and $\mathcal{F}_0$ of $\mathcal E$ and $\mathcal
F$, respectively, their orthogonal $\mathcal{E}_0^{\perp_\phi}$ and
$\mathcal{F}^{\top^\phi}$ are locally free sub-$\mathcal A$-modules
of varying finite rank of $\mathcal F$ and $\mathcal E$,
respectively.
\end{enumerate}
\end{definition}

\begin{definition}
The $\mathcal A$-pairing $(\mathcal{E}^\ast, \mathcal{E}; \phi)$,
where $\mathcal E$ is a free $\mathcal A$-module and such that for
every open $U\subseteq X$,
\[
\phi_U(\psi, r):= \psi_U(r),
\]
where $\psi\in \mathcal{E}^\ast(U):=
\mbox{H}om_{\mathcal{A}|_U}(\mathcal{E}|_U, \mathcal{A}|_U)$ and
$r\in \mathcal{E}(U)$, is called the \textbf{canonical
$\mathcal{A}$-pairing} of $\mathcal E$  and $\mathcal{E}^\ast$.
\end{definition}

For every open $U\subseteq X,$ let $(\varepsilon_i^U)$ and
$(\varepsilon_i^{\ast U})$ be the canonical bases (of sections) of
$\mathcal{E}(U)$ and $\mathcal{E}^\ast(U)$, where $\mathcal E$ and
$\mathcal{E}^\ast$ are \textit{canonically paired into} $\mathcal
A$. The family $\phi\equiv (\phi_U)_{X\supseteq U,\ open}$ such that
\[
\phi_U(\varepsilon_i^U):= \varepsilon_i^{\ast U}
\]
is an $\mathcal A$-isomorphism of $\mathcal E$ onto
$\mathcal{E}^\ast.$ Furthermore, the kernel of $\phi$ is exactly the
same as the left kernel of the canonical $\mathcal A$-pairing
$(\mathcal{E}^\ast, \mathcal{E}; \mathcal{A})$. Indeed, $\ker \phi=
0 = \mathcal{E}^{\ast \top}.$

\begin{theorem}\label{theo4}
Let $\mathcal E$ be a free $\mathcal A$-module of finite dimension.
The canonical pairing $((\mathcal{E}^\ast, \mathcal{E}; \nu);
\mathcal{A})$ is orthogonally convenient.
\end{theorem}

\begin{proof}
First, we notice by Theorem \ref{theorem2} that both kernels, i.e.
$(\mathcal{E}^\ast)^\top$ and $\mathcal{E}^\perp$, are $0$. Let
$\mathcal{E}_0$ be a free sub-$\mathcal A$-module of $\mathcal E$,
and consider the second map of Theorem \ref{theorem3}$:
\mathcal{E}_0^\perp\longrightarrow
(\mathcal{E}/\mathcal{E}_0)^\ast.$ It is an $\mathcal A$-isomorphism
\textit{into}, and we shall show that \textit{it is onto.} Fix an
open set $U$ in $X$, and let $\psi\in
(\mathcal{E}/\mathcal{E}_0)^\ast(U):=
\mbox{H}om_{\mathcal{A}|_U}\left((\mathcal{E}/\mathcal{E}_0)|_U,
\mathcal{A}|_U\right).$ Let us consider a family
$\overline{\psi}\equiv (\overline{\psi}_V)_{U\supseteq V,\ open}$
such that
\begin{equation}\begin{array}{ll}\overline{\psi}_V(r):= \psi_V(r+
\mathcal{E}_0(V)), & r\in \mathcal{E}(V).
\label{eq18}\end{array}\end{equation} It is easy to see that
$\overline{\psi}_V$ is $\mathcal{A}(V)$-linear for any open
$V\subseteq U$. Now, let $\{\rho^V_W\}$, $\{\overline{\rho}^V_W\}$
and $\{\tau^V_W\}$ be the restriction maps for the
(\textit{complete}) \textit{presheaves of sections} of $\mathcal E$,
$\mathcal{E}/\mathcal{E}_0$ and $\mathcal A$, respectively. The
restriction maps $\overline{\rho}^V_W$ are defined by setting
\[\begin{array}{ll}
\overline{\rho}^V_W(r+ \mathcal{E}_0(V)):= \rho^V_W(r)+
\mathcal{E}_0(W), & r\in \mathcal{E}(V)\end{array}.
\]
It clearly follows that
\begin{eqnarray*}
(\tau^V_W\circ \overline{\psi}_V)(r) & = & \tau^V_W(\psi_V(r+
\mathcal{E}_0(V)))\\ & = & \psi_W(\rho^V_W(r)+ \mathcal{E}_0(W))\\ &
= & \overline{\psi}_W(\rho^V_W(r))\\ & = & (\overline{\psi}_W\circ
\rho^V_W)(r),
\end{eqnarray*}
from which we deduce that
\[
\tau^V_W\circ \overline{\psi}_V= \overline{\psi}_W\circ \rho^V_W,
\]
which implies that
\[
\overline{\psi}\in \mbox{H}om_{\mathcal{A}|_U}(\mathcal{E}|_U,
\mathcal{A}|_U)=: \mathcal{E}^\ast(U).
\] Suppose $r\in
\mathcal{E}_0(V)$, where $V$ is open in $U$. Then
\[\overline{\psi}_V(r)= \psi_V(r+ \mathcal{E}_0(V))=
\psi_V(\mathcal{E}_0(V))= 0,\]therefore
\[\nu_V(\overline{\psi}|_V, \mathcal{E}_0(V))=
\overline{\psi}_V(\mathcal{E}_0(V))=0,
\]i.e. $\overline{\psi}\in
\mathcal{E}_0^\perp(U).$ We contend that $\overline{\psi}$ has the
given $\psi$ as image under the second map of Theorem
\ref{theorem3}, and this will show the ontoness of the map thereof
and that $\mathcal{E}_0^\perp$ is a free sub-$\mathcal A$-module of
$\mathcal{E}^\ast$.

Let us find the image of $\overline{\psi}$. Consider the pairing
$((\mathcal{E}/\mathcal{E}_0, \mathcal{E}_0^\perp; \Theta);
\mathcal{A})$ such that for any open $V\subseteq X$, we have
\[\Theta_V(\alpha, r+ \mathcal{E}_0(V)):= \nu_V(\alpha, r)=
\alpha_V(r),\]where $\alpha\in \mathcal{E}_0^\perp(V)\subseteq
\mathcal{E}^\ast(V),$ $r\in \mathcal{E}(V)$. Clearly, the left
kernel of this new pairing is $0$. For $\alpha= \overline{\psi}\in
\mathcal{E}_0^\perp(U)\subseteq \mathcal{E}^\ast(U),$ we have
\[\Theta_U(\overline{\psi}, r+ \mathcal{E}_0(U))=
\overline{\psi}_U(r)\]where $r\in \mathcal{E}(U)$, and the map
\[\overline{\Theta}_U: \mathcal{E}_0^\perp(U)\longrightarrow
(\mathcal{E}/\mathcal{E}_0)^\ast(U)\]given by
\[\overline{\psi}\longmapsto \overline{\Theta}_{U,
\overline{\psi}}\equiv \left((\overline{\Theta}_{U,
\overline{\psi}})_V\right)_{U\supseteq V,\ open}\] and such that for
any $r\in \mathcal{E}(V)$ \[(\overline{\Theta}_{U,
\overline{\psi}})_V(r+ \mathcal{E}_0(V)):=
{\Theta}_V(\overline{\psi}|_V, r+\mathcal{E}_0(V))=
\overline{\psi}_V(r)= \psi_V(r+ \mathcal{E}_0(V))\]is the image.
Thus the image of $\overline{\psi}$ is $\psi$, hence the map
$\mathcal{E}_0^\perp(U)\longrightarrow
(\mathcal{E}/\mathcal{E}_0)^\ast(U)$ is onto, and therefore an
$\mathcal{A}(U)$-isomorphism. Since $\mathcal{E}/\mathcal{E}_0$ is
free by Corollary \ref{corollary1}, so are
$(\mathcal{E}/\mathcal{E}_0)^\ast$ and $\mathcal{E}_0^\perp$ free.

Now, let $\mathcal{F}_0$ be a free sub-$\mathcal A$-module of
$\mathcal{E}^\ast\cong \mathcal E$ (cf. Mallios~\cite[p.298,
(5.2)]{mallios}); on considering $\mathcal{F}_0$ as a free
sub-$\mathcal A$-module of $\mathcal E$, according to all that
precedes above $\mathcal{F}_0^\perp$ is free in
$\mathcal{E}^\ast\cong \mathcal E$, and so the proof is finished.
\end{proof}

\begin{definition}
\emph{Let $\mathcal E$ and $\mathcal F$ be free $\mathcal
A$-modules. An $\mathcal A$-morphism $\phi\in \emph{H}om_{\mathcal
A}(\mathcal{E}, \mathcal{F})$ is called \textbf{free} if  Im $\phi$
is a free sub-$\mathcal A$-module of $\mathcal F$. The dimension of
Im $\phi$ is called the \textbf{rank} of $\phi$, and is denoted
\textbf{rank} $\phi.$}
\end{definition}

We may now state the counterpart of the \textit{fundamental theorem}
of the classical theory, see \cite[p. 54, Th\'eor\`{e}me
6.4]{chambadal}.

\begin{theorem}\label{theorem5}
Let $\phi\in \emph{H}om_\mathcal{A}(\mathcal{E}, \mathcal{F})$ be a
free $\mathcal A$-morphism mapping a free $\mathcal A$-module
$\mathcal E$ into a free $\mathcal A$-module $\mathcal F$. Then, the
rank of $\phi$ is finite if and only if the kernel of $\phi$ has
finite codimension in $\mathcal E$. Moreover, one has
\[
\emph{\mbox{rank}} \phi:= \dim \emph{\mbox{Im}} \phi=
\emph{\mbox{codim}}_\mathcal{E}\ker \phi.
\]
\end{theorem}

\begin{proof}
Corollary \ref{corollary1}(1) shows that the quotient free $\mathcal
A$-module $\mathcal{E}/\ker \phi$ is $\mathcal A$-isomorphic to Im
$\phi$.
\end{proof}

\begin{corollary}\label{corollary4}
Let $\mathcal A$ be a PID algebra sheaf and $\mathcal E$, $\mathcal
F$ free $\mathcal A$-modules. Then, if $\dim \mathcal E$ is finite,
every free $\mathcal A$-morphism $\phi\in
\emph{H}om_\mathcal{A}(\mathcal{E}, \mathcal{F})$ has finite rank,
and
\begin{equation}\label{eq8}
\emph{\mbox{rank}}(\phi)+ \dim \ker(\phi)= \dim \mathcal{E}.
\end{equation}
The formula above is called the \textbf{dimension formula.}
\end{corollary}

\begin{proof}
Indeed, given that every $\mathcal{A}(U),$ where $U$ is open in $X$,
is a PID algebra, it follows that $\ker (\phi_U)$ is a free
sub-$\mathcal{A}(U)$- module of the free $\mathcal{A}(U)$-module
$\mathcal{E}(U)$. By elementary module theory (see, for instance,
\cite[p. 173, Proposition 8.8]{adkins} or \cite[p. 105, Corollary
2]{blyth}), we have
\[
\dim \ker(\phi_U)+ \dim {\mbox{Im}}(\phi_U)= \dim \mathcal{E}(U).
\]
Since for any subsets $U$ and $V$ of $X$, $\dim \ker(\phi_U)= \dim
\ker(\phi_V)$, it follows that $\ker(\phi)$ is a free sub-$\mathcal
A$-module of $\mathcal E$, and therefore
\[
\dim \ker(\phi)+ \dim {\mbox{Im}}(\phi)= \dim \mathcal{E},
\]
or
\[
\dim \ker(\phi)+ {\mbox{rank}}(\phi)= \dim \mathcal{E}.
\]
\end{proof}

\begin{theorem}\label{theorem7}
Let $(\mathcal{E}^\ast, \mathcal{E}; \mathcal{A})$ be the canonical
free $\mathcal A$-pairing, and $\mathcal F$ a free sub-$\mathcal
A$-module of $\mathcal{E}^\ast$. $\mathcal F$ has finite dimension
if and only if $\mathcal{F}^\top$ has finite codimension in
$\mathcal E$; moreover, one has
\[\begin{array}{ll}
\dim \mathcal{F}= \emph{\mbox{codim}}_\mathcal{E}\mathcal{F}^\top; &
(\mathcal{F}^\top)^\perp= \mathcal F. \end{array}
\]
\end{theorem}

\begin{proof}
The case $\mathcal{F}=0$ is trivial.

Suppose that $\mathcal F$ has finite dimension; let $U$ be an open
subset of $X$, $(e_1^{U\ast}, \ldots, e_n^{U\ast})$ a canonical
(local) gauge of $\mathcal F$ (cf. \cite[p. 291, (3.11) along with
p. 301, (5.17) and (5.18)]{mallios}), and $\phi\in
\emph{H}om_\mathcal{A}(\mathcal{E}, \mathcal{A}^n)$ be such that if
$s\in \mathcal{E}(U),$
\[
\phi_U(s):= (e_1^{U\ast}(s), \ldots, e_n^{U\ast}(s)).
\]
It is clear that $\phi$ is indeed an $\mathcal A$-morphism of
$\mathcal E$ into $\mathcal{A}^n$ whose kernel is
$\mathcal{F}^\top$, which is a free sub-$\mathcal A$-module of
$\mathcal E$ for the simple reason that \textit{canonical free
$\mathcal A$-pairings are orthogonally convenient,} see Theorem
\ref{theo4}. It is also clear that Im $\phi$ is $\mathcal
A$-isomorphic to the free $\mathcal A$-module $\mathcal{A}^n$; thus,
by Theorem \ref{theorem5}, one has
\begin{equation}\label{eq6}
{\mbox{rank}}(\phi):= {\mbox{codim}}_\mathcal{E}\mathcal{F}^\top=
\dim \mathcal{F}.
\end{equation}
According to Theorem \ref{theorem6}(3), $(\mathcal{F}^\top)^\perp$
has finite dimension, and
\begin{equation}\label{eq7}
\dim (\mathcal{F}^\top)^\perp=
\mbox{codim}_\mathcal{E}\mathcal{F}^\top.
\end{equation}
Since $\mathcal F$ is contained in $(\mathcal{F}^\top)^\perp$, we
deduce from (\ref{eq6}) and (\ref{eq7}) that
\[
\mathcal{F}= (\mathcal{F}^\top)^\perp.
\]

Conversely, suppose that $\mathcal{F}^\top$ has finite codimension
in $\mathcal E$; then $(\mathcal{F}^\top)^\perp$ has finite
dimension, and thus $\mathcal F$ as well, as $\mathcal F$ is
contained in $(\mathcal{F}^\top)^\perp.$
\end{proof}

\section{Biorthogonality with respect to arbitrary $\mathcal A$-bilinear forms}

In this section, we investigate the results of the previous sections
in a more general setting, that is, $\mathcal A$-pairings defined by
arbitrary $\mathcal A$-bilinear morphisms. The section ends with the
Witt's hyperbolic decomposition theorem for $\mathcal A$-modules.

\begin{definition}\label{definition1}
\emph{Let $((\mathcal{F}, \mathcal{E}; \phi); \mathcal{A})$ be an
{$\mathcal A$-pairing}. The $\mathcal A$-bilinear morphism $\phi$ is
said to be \textbf{non-degenerate} if $\mathcal{E}^{\perp_\phi}=
\mathcal{F}^{\top^\phi}=0,$ and degenerate otherwise. For every open
$U\subseteq X$,
\[
\mathcal{E}^{\perp_\phi}(U):= \{t\in \mathcal{F}(U):\
\phi_V(\mathcal{E}(V), t|_V)=0,\ \mbox{for any open $V\subseteq
U$}\}.
\]
Similarly,
\[
\mathcal{F}^{\top^\phi}(U):= \{s\in \mathcal{E}(U): \phi_V(s|_V,
\mathcal{F}(V))=0,\ \mbox{for any open $V\subseteq U$}\}.
\]
}
\end{definition}

\begin{definition}
\emph{Let $\mathcal E$ and $\mathcal F$ be $\mathcal A$-modules
$($not necessarily free$)$ and $\phi: \mathcal{E}\oplus
\mathcal{F}\longrightarrow \mathcal{A}$ and $\mathcal A$-bilinear
morphism. The $\mathcal A$-morphism
\[
\phi^R\in \emph{H}om_\mathcal{A}(\mathcal{F},
\emph{H}om_\mathcal{A}(\mathcal{E}, \mathcal{A}))
\]
such that, for any open subset $U\subseteq X$ and sections $t\in
\mathcal{F}(U)$ and $s\in \mathcal{E}(V),$ where $V\subseteq U$ is
open,
\[
\phi^R_U(t)(s)\equiv (\phi^R)_U(t)(s):= \phi_V(s, t|_V)
\]
is called the \textbf{right insertion $\mathcal A$-morphism}
associated with $\phi.$ Similarly, for every open subset $U\subseteq
X$ and sections $s\in \mathcal{E}(U)$ and $t\in \mathcal{F}(V),$
where $V$ is open in $U,$
\[
\phi^L_U(s)(t)\equiv (\phi^L)_U(s)(t):= \phi_V(s|_V, t)
\]
defines an $\mathcal A$-morphism, denoted $\phi^L$, of $\mathcal E$
into $\mathcal{F}^\ast$, i.e.,
\[
\phi^L\in \emph{H}om_\mathcal{A}(\mathcal{E},
\emph{H}om_\mathcal{A}(\mathcal{F}, \mathcal{A})).
\]
The $\mathcal A$-morphism $\phi^L$ is called the \textbf{left
insertion $\mathcal A$-morphism} associated with $\phi.$}
\end{definition}

It is clear in the light of Definition \ref{definition1} that if the
$\mathcal A$-bilinear morphism $\phi: \mathcal{E}\oplus
\mathcal{F}\longrightarrow \mathcal A$ is non-degenerate, then both
insertion $\mathcal A$-morphisms $\phi^R$ and $\phi^L$ are
injective. Moreover, if $\mathcal E$ and $\mathcal F$ are free
$\mathcal A$-modules of finite dimension, then
\[
\mathcal{E}= \mathcal{F}
\]
within an $\mathcal A$-isomorphism.

In other to differentiate a canonical $\mathcal A$-pairing
$(\mathcal{E}^\ast, \mathcal{E}; \nu)$ from an arbitrary $\mathcal
A$-pairing $(\mathcal{F}, \mathcal{E}; \phi)$, in which $\mathcal F$
may still be the dual $\mathcal A$-module $\mathcal{E}^\ast$, we
will adopt the following notation for the orthogonal sub-$\mathcal
A$-module associated with a given sub-$\mathcal A$-module:

\begin{definition}
\emph{Let $(\mathcal{F}, \mathcal{E}; \phi)$ be an $\mathcal A$-pairing
and $\mathcal G$ a sub-$\mathcal A$-module of $\mathcal E$. The
sub-$\mathcal A$-module $\mathcal{G}^{\perp_\phi}\subseteq \mathcal
F$ such that, for every open $U\subseteq X,$
\[
\mathcal{G}^{\perp_\phi}:= \{t\in \mathcal{F}(U):\
\phi_V(\mathcal{G}(V), t|_V)=0, \mbox{for any open $V\subseteq U$}\}
\]
is called the \textbf{$\phi$-orthogonal} (or just
\textbf{orthogonal} if there is no confusion to fear) \textbf{of}
$\mathcal G$ in $\mathcal F$. Similarly, one defines the
$\phi$-orthogonal sub-$\mathcal A$-module $\mathcal{H}^{\top^\phi}$
of a given sub-$\mathcal A$-module $\mathcal H$ of $\mathcal F$, viz
\[
\mathcal{H}^{\top^\phi}:= \{s\in \mathcal{E}(U):\ \phi_V(s|_V,
\mathcal{H}(V))=0, \mbox{for any open $V\subseteq U$}\}.
\]}
\end{definition}

While the notion of orthogonality with respect to arbitrary
$\mathcal A$-bilinear forms generalizes orthogonality in canonical
$\mathcal A$-pairings, the former may relate with the latter through
the following lemma.

\begin{lemma}\label{lemma1}
Let $(\mathcal{F}, \mathcal{E}; \phi)$ be a free $\mathcal
A$-pairing, $\mathcal G$ and $\mathcal H$ free sub-$\mathcal
A$-modules of $\mathcal E$ and $\mathcal F$, respectively. Then,
\begin{equation}\label{eq9}
\mathcal{G}^{\perp_\phi}\simeq (\phi^L(\mathcal{G}))^\top,
\end{equation}
and
\begin{equation}\label{eq10}
\mathcal{H}^{\top^\phi}\simeq (\phi^R(\mathcal{H}))^\top.
\end{equation}
\end{lemma}

\begin{proof}
Let $U$ be an open subset of $X$. Since $\mathcal G$ is free, it is
clear that for a section $t\in \mathcal{F}(U)$ to be in
$\mathcal{G}^{\perp_\phi}$ it is necessary and sufficient that
\[
\phi_U(\mathcal{G}(U), t)=0.
\]
But
\[
(\phi^L_U(\mathcal{G}(U)))^\top= \{t\in \mathcal{F}(U):\
\phi^L_U(\mathcal{G}(U))(t):= \phi_U(\mathcal{G}(U), t)=0\},
\]
therefore (\ref{eq9}) holds as required.

In a similar way, one shows (\ref{eq10}).
\end{proof}

The case where $((\mathcal{F}, \mathcal{E}; \phi); \mathcal{A})$ is
an \textit{orthogonally convenient pairing} and $\phi$ is
\textit{degenerate} is interesting, for it yields the following
result.

\begin{proposition}\label{proposition1}
Given an orthogonally convenient pairing $((\mathcal{F},
\mathcal{E}; \phi); \mathcal{A})$, where $\mathcal E$ and
$\mathcal F$ are free $\mathcal A$-modules of finite dimension,
the free quotient $\mathcal A$-modules
$\mathcal{E}/\mathcal{F}^{\top^\phi}$ and
$\mathcal{F}/\mathcal{E}^{\perp_\phi}$ have the same dimension,
i.e.
\[
\mathcal{E}/\mathcal{F}^{\top^\phi}=
\mathcal{F}/\mathcal{E}^{\perp_\phi}
\]
within an $\mathcal A$-isomorphism.
\end{proposition}

\begin{proof}
Since $((\mathcal{F}, \mathcal{E}; \phi); \mathcal{A})$ is
\textit{orthogonally convenient,} kernels
$\mathcal{E}^{\perp_\phi}$ and $\mathcal{F}^{\top^\phi}$ are free
sub-$\mathcal A$-modules of $\mathcal F$ and $\mathcal E$,
respectively. By \cite{malliosntumbaqm2}, it follows that the
quotient $\mathcal A$-modules
$\mathcal{E}/\mathcal{F}^{\top^\phi}$ and
$\mathcal{F}/\mathcal{E}^{\perp_\phi}$ are free, and for any open
subset $U$ of $X$,
\[
(\mathcal{E}/\mathcal{F}^{\top^\phi})(U)=
\mathcal{E}(U)/\mathcal{F}^{\top^\phi}(U)=
\mathcal{E}(U)/\mathcal{F}(U)^{\top^\phi}
\]
and
\[
(\mathcal{F}/\mathcal{E}^{\perp_\phi})(U)=
\mathcal{F}(U)/\mathcal{E}^{\perp_\phi}(U)=
\mathcal{F}(U)/\mathcal{E}(U)^{\perp_\phi}
\]
within $\mathcal{A}(U)$-isomorphism. Clearly, for a fixed open
$U\subseteq X$, if $s\in \mathcal{E}(U)$ and $t, t_1\in
\mathcal{F}(U)$ such that $t-t_1\in \mathcal{E}^{\perp_\phi}(U)$,
then
\[
\phi_U(s, t)= \phi_U(s, t_1).
\]

In the same vein, if $s=s_1 \mod \mathcal{F}^{\top^\phi}(U)$ and
$t=t_1 \mod \mathcal{E}^{\perp_\phi}(U),$ then
\[
\phi_U(s, t)= \phi_U(s_1, t_1).
\]

Now, let us consider the $\mathcal A$-bilinear morphism
\[
\overline{\phi}\equiv (\overline{\phi}_U)_{X\supseteq U,\
open}\equiv ((\overline{\phi})_U)_{X\supseteq U,\ open}:
\mathcal{E}/\mathcal{F}^{\top^\phi}\oplus
\mathcal{F}/\mathcal{E}^{\perp_\phi}\longrightarrow \mathcal{A},
\]
induced by the $\mathcal A$-bilinear morphism $\phi,$ which is
such that, for any open $U\subseteq X$ and sections
$\overline{s}:= \mbox{cl}(s) \mod\mathcal{F}^{\top^\phi}(U),$
$\overline{t}:= \mbox{cl}(t) \mod \mathcal{E}^{\perp_\phi}(U)$
($\mbox{cl}(s)$ stand for the \textit{equivalence class
containing} $s$), one has
\[
\overline{\phi}_U(\overline{s}, \overline{t}):= \phi_U(s, t).
\]
It is clear that $\overline{\phi}_U(\overline{s}, \overline{t})=0$
for any $\overline{s}\in (\mathcal{E}/\mathcal{F}^{\top^\phi})(U)=
\mathcal{E}(U)/\mathcal{F}^{\top^\phi}(U)$ is equivalent to
$\phi_U(s, t)=0$ for any $s\in \mathcal{E}(U);$ therefore $t\in
\mathcal{E}^{\perp_\phi}(U)=0$  and hence $\overline{t}=0.$ This
implies that
$(\mathcal{E}/\mathcal{F}^{\top^\phi})^{\perp_\phi}=0.$ Similarly,
that $\overline{\phi}_U(\overline{s}, \overline{t})=0$ for any
$\overline{t}\in (\mathcal{F}/\mathcal{E}^{\perp_\phi})(U)=
\mathcal{F}(U)/\mathcal{E}^{\perp_\phi}(U)$ is equivalent to
$\overline{s}=0,$ from which we deduce that
$(\mathcal{F}/\mathcal{E}^{\perp_\phi})^{\top^\phi}=0.$ Hence,
$\overline{\phi}$ is non-degenerate; so
\[
\mathcal{E}/\mathcal{F}^{\top^\phi}=
\mathcal{F}/\mathcal{E}^{\perp_\phi}
\]
within an $\mathcal A$-isomorphism.
\end{proof}

Based on Proposition \ref{proposition1}, we make the following
definition.

\begin{definition}
\emph{Let $\mathcal{E}$, $\mathcal F$ be free $\mathcal A$-modules,
and $((\mathcal{F}, \mathcal{E}; \phi); \mathcal{A})$ an
orthogonally convenient $\mathcal A$-pairing. The dimension of the
free $\mathcal A$-module $\mathcal{E}/\mathcal{F}^{\top^\phi}\cong
\mathcal{F}/\mathcal{E}^{\perp_\phi}$ is called the \textbf{rank of
$\phi$.}}
\end{definition}

\begin{theorem}\label{theorem8}
Let $\mathcal A$ be a PID algebra sheaf, $(\mathcal{F}, \mathcal{E};
\phi)$ an orthogonally convenient $\mathcal A$-pairing with $\dim
\mathcal E$ and $\dim \mathcal F$ finite. Moreover, let $\phi^L$ and
$\phi^R$ be the left and right insertion $\mathcal A$-morphisms
associated with $\phi.$ Then,
\begin{enumerate}
\item [{$(1)$}] For every free sub-$\mathcal A$-modules $\mathcal G$
and $\mathcal H$ of $\mathcal E$ and $\mathcal F$, respectively, one
has
\begin{enumerate}
\item [{$1.1)$}] $\phi^L(\mathcal{G})\simeq
(\mathcal{G}^{\perp_\phi})^\perp$ and $\phi^R(\mathcal{H})\simeq
(\mathcal{H}^{\top^\phi})^\perp.$
\item [{$1.2)$}] $\dim \phi^L(\mathcal{G})=
\emph{\mbox{codim}}_\mathcal{F}\mathcal{G}^{\perp_\phi}$ and $\dim
\phi^R(\mathcal{H})=
\emph{\mbox{codim}}_\mathcal{E}\mathcal{H}^{\top^\phi}.$
\end{enumerate}
\item [{$(2)$}] $\mathcal A$-morphisms $\phi^L$ and $\phi^R$ have
the same rank:
\begin{equation}\label{eq15}
\emph{\mbox{rank}}(\phi^L)= \emph{\mbox{rank}}(\phi^R),
\end{equation}
which is the rank of $\phi.$
\end{enumerate}
\end{theorem}

\begin{proof}
\textit{Assertion $(1)$.} Since $(\mathcal{F}, \mathcal{E}; \phi)$
is orthogonally convenient, the sub-$\mathcal A$-module
$\mathcal{G}^{\perp_\phi}$ is free, and thus
\[
\mathcal{G}^{\perp_\phi}(U)\simeq \mathcal{G}(U)^{\perp_\phi}
\]
for every open $U\subseteq X$. By Lemma \ref{lemma1},
\[
\mathcal{G}^{\perp_\phi}= (\phi^L(\mathcal{G}))^\top
\]
within an $\mathcal A$-isomorphism. Applying Theorem \ref{theorem7},
and since $\dim \mathcal F$ is finite, we have
\[
(\mathcal{G}^{\perp_\phi})^\perp= \phi^L\mathcal G
\]
within an $\mathcal A$-isomorphism. By the same theorem along with
Theorem \ref{theorem6}, it follows that
\[
\dim \mathcal{G}^{\perp_\phi}+ \dim \phi^L\mathcal{G}= \dim
{\mathcal F},
\]
from which we deduce that
\[
\dim \phi^L\mathcal{G}=
\mbox{codim}_\mathcal{F}\mathcal{G}^{\perp_\phi}.
\]

In particular,
\begin{equation}\label{eq11}
\mbox{rank}(\phi^L)=
\mbox{codim}_\mathcal{F}\mathcal{E}^{\perp_\phi}.
\end{equation}

In a similar way, one shows the claims related to the induced
$\mathcal A$-morphism $\phi^R$ by using the fact that $\dim \mathcal
E$ is finite. The analog of (\ref{eq11}) is
\begin{equation}\label{eq12}
\mbox{rank}(\phi^L)=
\mbox{codim}_\mathcal{E}\mathcal{F}^{\top^\phi}.
\end{equation}

\textit{Assertion $(2)$.} That
\[\begin{array}{lll}
\ker(\phi^L)\simeq \mathcal{E}^{\perp_\phi} & \mbox{and} &
\ker(\phi^R)\simeq \mathcal{F}^{\top^\phi}\end{array}
\]
is immediate. Applying the \textit{dimension formula} (Corollary
\ref{corollary4}), we obtain
\begin{equation}\label{eq13}
\mbox{rank}(\phi^R):= \dim \phi^R(\mathcal{F})= \dim \mathcal{F}-
\dim \mathcal{E}^{\perp_\phi}=
\mbox{codim}_\mathcal{F}\mathcal{E}^{\perp_\phi},
\end{equation}
and
\begin{equation}\label{eq14}
\mbox{rank}(\phi^L):= \dim \phi^L(\mathcal{E})= \dim \mathcal{E}-
\dim \mathcal{F}^{\top^\phi}=
\mbox{codim}_\mathcal{E}\mathcal{F}^{\top^\phi}.
\end{equation}
From (\ref{eq11}), (\ref{eq12}), (\ref{eq13}) and (\ref{eq14}), one
gets (\ref{eq15}).
\end{proof}

\begin{corollary}
Let $\mathcal A$ be a PID algebra sheaf and $(\mathcal{F},
\mathcal{E}; \phi)$ an orthogonally convenient $\mathcal
A$-pairing with free $\mathcal A$-modules $\mathcal E$ and
$\mathcal F$ both of finite dimension.
\begin{enumerate}
\item [{$(1)$}] For every free sub-$\mathcal A$-modules $\mathcal
G$ and $\mathcal H$ of $\mathcal E$ and $\mathcal F$,
respectively, one has
\begin{enumerate}
\item [{$1.1)$}] $\dim \mathcal{G}^{\perp_\phi}\geq \dim
\mathcal{F}- \dim \mathcal{G}$ and $\dim
\mathcal{H}^{\top^\phi}\geq \dim \mathcal{E}- \dim \mathcal{H}$
\item [{$1.2)$}] $(\mathcal{G}^{\perp_\phi})^{\top^\phi}\supseteq
\mathcal{G}$ and $(\mathcal{H}^{\top^\phi})^{\perp_\phi}\supseteq
\mathcal{H}.$
\end{enumerate}
\item [{$(2)$}] If $\phi$ is nondegenerate, then
\begin{enumerate}
\item [{$2.1)$}] $\dim \mathcal{G}^{\perp_\phi}+ \dim \mathcal{G}=
\dim \mathcal{F}= \dim \mathcal{E}= \dim \mathcal{H}^{\top^\phi}+
\dim \mathcal{H}$ \item [{$2.2)$}]
$(\mathcal{G}^{\perp_\phi})^{\top^\phi}\simeq \mathcal{G}$ and
$(\mathcal{H}^{\top^\phi})^{\perp_\phi}\simeq \mathcal{H}.$
\end{enumerate}
\end{enumerate}
\end{corollary}

\begin{proof} \textit{Assertion $(1).$}
Theorem \ref{theorem8} shows that
\[
\dim \phi^L(\mathcal{G})=
\mbox{codim}_\mathcal{F}\mathcal{G}^{\perp_\phi}= \dim
\mathcal{F}- \dim \mathcal{G}^{\perp_\phi}.
\]
On the other hand, by virtue of Corollary \ref{corollary4}, one
has
\[
\dim \phi^L(\mathcal{G})= \dim \mathcal{G}- \dim (\ker \phi^L\cap
\mathcal{G}).
\]
It follows, in particular, that
\[
\dim \mathcal{G}\geq \dim \phi^L(\mathcal{G}),
\]
from which we have
\[
\dim \mathcal{G}^{\perp_\phi}\geq \dim \mathcal{F}- \dim
\mathcal{G}.
\]
Likewise, one shows the second inequality of $1.1)$.

\textit{Assertion $(2).$} If $\phi$ is nondegenerate, $\dim
\mathcal{E}= \dim \mathcal F$; therefore $\phi^L$ is an $\mathcal
A$-isomorphism of $\mathcal E$ onto $\mathcal{F}^\ast$. Thus,
$\dim \phi^L(\mathcal{G})= \dim \mathcal G$, and
\[
\dim \mathcal{G}^{\perp_\phi}= \dim \mathcal{F}- \dim \mathcal{G}.
\]
Likewise, one has
\[
\dim\mathcal{H}^{\top^\phi}= \dim \mathcal{E}- \dim \mathcal{H}.
\]

Applying relation $2.1)$ to the free sub-$\mathcal A$-modules
$\mathcal G$ and $\mathcal{G}^{\perp_\phi}$ of $\mathcal E$ and
$\mathcal F$, respectively, we see that
\[
\dim (\mathcal{G}^{\perp_\phi})^{\top^\phi}= \dim \mathcal{G}.
\]
Since $\mathcal G$ is contained in
$(\mathcal{G}^{\perp_\phi})^{\top^\phi},$ it follows that
\[
(\mathcal{G}^{\perp_\phi})^{\top^\phi}= \mathcal G
\]
within an $\mathcal A$-isomorphism. In a similar way, we show that
$(\mathcal{H}^{\top^\phi})^{\perp_\phi}= \mathcal H$ within an
$\mathcal A$-isomorphism.
\end{proof}

We will soon turn to the hyperbolic decomposition theorem for
$\mathcal A$-modules. However, the so-called theorem requires some
preparations.

\begin{lemma}\label{lemma2}
Let $(\mathcal{E}, \phi)$ be a symplectic $\mathcal A$-module, $U$
an open subset of $X$ and $(r_1, \ldots, r_n)\subseteq
\mathcal{E}(U)$ an arbitrary $($local$)$ gauge of $\mathcal E$. For
any $r\equiv r_i$, $1\leq i\leq n$, there exists a nowhere-zero
section $s\in \mathcal{E}(U)$ such that $\phi_U(r, s)$ is nowhere
zero.
\end{lemma}

\begin{proof}
Without loss of generality, assume that $r_1=r$. On the other hand,
since the induced $\mathcal A$-morphism $\widetilde{\phi}\in
\emph{H}om_\mathcal{A}(\mathcal{E}, \mathcal{E}^\ast)$ is one-to-one
and both $\mathcal E$ and $\mathcal{E}^\ast$ have the same finite
rank, it follows that the matrix $D$ representing $\phi_U$ (see also
\cite[p. 357, Theorem 2.21, along with p. 356, Definition
2.19]{adkins} or \cite[p. 343, Proposition 20.3]{chambadal}), with
respect to the basis $(r_1, \ldots, r_n)$, has a
\textit{nowhere-zero determinant}; so since
\[\det D= \sum_{i=1}^n (-1)^{1+i}\phi(r_1, r_i)\det D_{1i}=
\phi(r_1, \sum_{i=1}^n (-1)^{1+i}\det D_{1i}r_i),\] where $D_{1i}$
is the minor of the corresponding $\phi(r_1, r_i)$, and $\det D$
nowhere zero, we thus have a section $s:= \sum_{i=1}^n
(-1)^{1+i}\det D_{1i}r_i\in \mathcal{E}(U)$ such that $\phi(r, s)$
is nowhere zero.
\end{proof}

For the purpose of Theorem \ref{theorem9} below, we require the
following result, see \cite{malliosntumba2}.

\begin{theorem}
Let $\mathcal E$ be a free $\mathcal A$-module of finite rank,
equipped with an $\mathcal A$-bilinear morphism $\phi:
\mathcal{E}\oplus \mathcal{E}\longrightarrow \mathcal{A}.$ Then,
every non-isotropic free sub-$\mathcal A$-module $\mathcal F$ of
$\mathcal E$ is a direct summand of $\mathcal E$; viz.
\[
\mathcal{E}= \mathcal{F}\bot \mathcal{F}^{\perp_\phi}.
\]
\end{theorem}

So, we have

\begin{theorem}\label{theorem9}$($\textbf{Hyperbolic Decomposition
Theorem}$)$
Let $X$ be a connected topological space, $\mathcal A$ a PID algebra
sheaf on $X$, $(\mathcal E, \mathcal{E}; \phi)$ an orthogonally
convenient self-pairing, where $\mathcal E$ is a $($free-$)$
$\mathcal A$-module on $X$ and $\phi: \mathcal{E}\oplus
\mathcal{E}\longrightarrow \mathcal A$ a non-degenerate
skew-symmetric $\mathcal A$-bilinear morphism. Then, if $\mathcal F$
is a totally isotropic $($free$)$ sub-$\mathcal A$-module of rank
$k$, there is a \textsf{non-isotropic sub-$\mathcal A$-module
$\mathcal H$} of $\mathcal E$ of the form
\[
\mathcal{H}:= \mathcal{H}_1\bot\cdots \bot \mathcal{H}_k,
\]
where if $(r_{1, U}, \ldots, r_{k, U})$ is a basis of
$\mathcal{F}(U)$ $($with $U$ an open subset of $X)$, then $r_{i,
u}\in \mathcal{H}_i(U)$ for $1\leq i\leq k.$
\end{theorem}

\begin{proof}
Suppose that $k=1,$ i.e. $\mathcal{F}= \mathcal A$, within an
$\mathcal{A}$-isomorphism. If $\mathcal{F}(X)= [r_X]$ with $r_X\in
\mathcal{E}(X)$ a nowhere-zero section, then for every open
$U\subseteq X$, $r_U\equiv {r_X}|_U$ generates the
$\mathcal{A}(U)$-module $\mathcal{F}(U)$. Since $\phi_X$ is
non-degenerate, by Lemma \ref{lemma2}, there exists a nowhere-zero
section $s_X\in \mathcal{E}(X)$ such that $\phi_X(r_X, s_X)$ is
nowhere zero. The correspondence
\[
U\longmapsto \mathcal{H}(U):= [r_U, s_U]\equiv [{r_X}|_U, {s_X}|_U],
\]
where $U$ runs over the open sets in $X$, along with the obvious
restriction maps yields a \textit{complete presheaf of
$\mathcal{A}(U)$-modules} on $X$. Clearly, the pair $(\mathcal{H},
\overline{\phi})$, where $\overline{\phi}$ is the $\mathcal
A$-bilinear $\overline{\phi}: \mathcal{H}\oplus
\mathcal{H}\longrightarrow \mathcal A$ such that
\[
(r, s)\longmapsto \overline{\phi}_U(r, s):= \phi_U(r, s),
\]
where $r, s\in \mathcal{H}(U)$, is \textit{non-isotropic.} Hence,
the theorem holds for the case $k=1.$ Let us now proceed by
induction to $k>1.$ To this end, put $\mathcal{F}_{k-1}\simeq
\mathcal{A}^{k-1}$ and $\mathcal{F}_k:= \mathcal{F}\simeq
\mathcal{A}^k.$ Then, $\mathcal{F}_{k-1}\varsubsetneqq
\mathcal{F}_k,$ so $\mathcal{F}_k^{\perp_\phi}\varsubsetneqq
\mathcal{F}^{\perp_\phi}_{k-1}.$ Since \textit{orthogonal of free
sub-$\mathcal{A}$-modules in an orthogonally convenient
$\mathcal{A}$-module are free sub-$\mathcal{A}$-modules,} the
inclusion $\mathcal{F}_k^{\perp_\phi}\varsubsetneqq
\mathcal{F}^{\perp_\phi}_{k-1}$ implies that, if
$\mathcal{F}^{\perp_\phi}_{k-1}\simeq \mathcal{A}^m$ and
$\mathcal{F}_k^{\perp_\phi}\simeq \mathcal{A}^n$ with $n<m$, then
there exists a sub-$\mathcal A$-module $\mathcal{G}\subseteq
\mathcal{F}^\perp_{k-1}$ such that $\mathcal{G}\simeq
\mathcal{A}^{m-n}.$ For every open $U\subseteq X$, pick a
nowhere-zero section $s_{k, U}\in \mathcal{G}(U),$ and put
$\mathcal{H}_k(U)= [r_{k, U}, s_{k, U}].$ The correspondence
\[U\longmapsto \mathcal{H}_k(U),\]where $U$ is open in $X$, along
with the obvious restriction maps, is a complete presheaf of
$\mathcal{A}(U)$-modules. Since $\phi_U(r_{i, U}, s_{k, U})=0$ for
$1\leq i\leq k-1,$ $\phi_U(r_{k, U}, s_{k, U})$ is nowhere zero.
Hence, $\mathcal{H}_k(U)$ is a non-isotropic $\mathcal{A}(U)$-plane
containing $r_{k, U}.$ By Theorem \ref{theorem9} $\mathcal{E}=
\mathcal{H}_k\bot \mathcal{H}_k^{\perp_\phi}.$ Since $r_{k, U},
s_{k, U}\in \mathcal{F}^{\perp_\phi}_{k-1}(U),$
$\mathcal{H}_k(U)\subseteq \mathcal{F}_{k-1}^{\perp_\phi}(U)$ for
every open $U\subseteq X$; so $\mathcal{H}_k\subseteq
\mathcal{F}^{\perp_\phi}_{k-1},$ which in turn implies that
$\mathcal{F}_{k-1}\subseteq \mathcal{H}^{\perp_\phi}_k.$ Apply an
inductive argument to $\mathcal{F}_{k-1}$ regarded as a
sub-$\mathcal{A}$-module of the non-isotropic $\mathcal{A}$-module
$\mathcal{H}^{\perp_\phi}_k$.
\end{proof}

\noindent Patrice P. Ntumba\\{Department of Mathematics and Applied
Mathematics}\\{University of Pretoria}\\ {Hatfield 0002, Republic of
South Africa}\\{Email: patrice.ntumba@up.ac.za}

\noindent Adaeze C. Orioha\\{Department of Mathematics and Applied
Mathematics}\\{University of Pretoria}\\ {Hatfield 0002, Republic of
South Africa}\\{Email: adaezeanyaegbunam@rocketmail.com}


\addcontentsline{toc}{section}{REFERENCES}
\begin{thebibliography}{99}
\bibitem{adkins} W.A. Adkins, S.H. Weintraub: \textit{Algebra. An Approach via Module
Theory.} Springer-Verlag New York, 1992.
\bibitem{artin} E. Artin: \textit{Geometric Algebra.} Interscience
Publishers, New York, 1988.
\bibitem{blyth} T.S. Blyth: \textit{Module Theory. An Approach to Linear Algebra. Second
edition}. Oxford Science Publications, Clarendon Press, Oxford,
1990.
\bibitem{chambadal} L. Chambadal, J.L. Ovaert: \textit{Alg\`{e}bre
Lin\'{e}aire et Alg\`{e}bre Tensorielle.} Dunod, Paris, 1968.
\bibitem{crumeyrolle} A. Crumeyrolle: \textit{Orthogonal and Symplectic Clifford Algebras. Spinor
Structures.} Kluwer Academic Publishers, Dordrecht, 1990.
\bibitem{deheuvels} R. Deheuvels: \textit{Formes quadratiques et groupes classiques.}
Presses Universitaires de France, 1981.
\bibitem{lang} S. Lang: \textit{Algebra. Revised Third
Edition}. Springer, 2002.
\bibitem{mallios1997} A. Mallios: \textit{On an Axiomatic Treatement of Differential Geometry via}
\linebreak \textit{Vector Sheaves. Applications.} Mathematica
Japonica (Intern. Plaza) \textbf{48} (1998), 93-180. (invited paper)
\bibitem{mallios} A. Mallios: \textit{Geometry of Vector Sheaves. An Axiomatic Approach to Differential Geometry.
Volume $I$: Vector Sheaves. General Theory.} Kluwer Academic
Publishers, Dordrecht, 1998.
\bibitem{malliosvolume2} A. Mallios: \textit{Geometry of Vector Sheaves}. An Axiomatic Approach to Differential Geometry.
Volume $II$: Geometry. Examples and Applications. Kluwer Academic
Publishers. Netherlands, 1998.
\bibitem{modern} A. Mallios: \textit{Modern Differential Geometry in
Gauge Theories: Maxwell Fields, vol. I. Yang-Mills Fields, vol. II.}
Birkh\"{a}user, Boston, 2006/2007.
\bibitem{malliosntumbaqm1} A. Mallios, P.P. Ntumba: \textit{Pairings of
sheaves of $\mathcal A$-modules.} Quaestiones Mathematicae
\textbf{31}(2008), 397-414.
\bibitem{malliosntumba2} A. Mallios, P.P. Ntumba: \textit{On a sheaf-theoretic
version of the Witt's decomposition theorem. A Lagrangian
perspective.} Rend. Circ. Mat. Palermo \textbf{58} (2009),155--168.
\bibitem{darboux} A. Mallios, P.P. Ntumba:
\textit{Fundamentals for symplectic $\mathcal{A}$-modules. Affine
Darboux theorem.} Rend. Circ. Mat. Palermo \textbf{58} (2009), 169-
198.
\bibitem{malliosntumbaqm2} A. Mallios, P.P. Ntumba: \textit{On the Second
Isomorphism Theorem for $\mathcal A$-Modules, and $\ldots$
``\textit{all that}"} (under refereeing)
\bibitem{cartandieudonne} P. P. Ntumba: \textit{Cartan-Dieudonn\'{e}
Theorem for $\mathcal A$-Modules} (to appear in Mediterr. J. Math.
\textbf{7} (2010))
\bibitem{omeara} O.T. O'Meara: \textit{Symplectic Groups.} American
Mathematical Society, Providence, Rhode Island, 1978.

\end{thebibliography}
\end{document}